\documentclass{amsart}
\usepackage{amsfonts,amssymb,amscd,amsmath,enumerate,verbatim,calc}
\usepackage[all]{xy}

\newcommand{\CM}{Cohen-Macaulay}

\newcommand{\m}{\mathfrak{m} }

\newcommand{\q}{\mathfrak{q} }
\newcommand{\p}{\mathfrak{p} }

\newcommand{\D}{\mathbb{D} }

\newcommand{\K}{\mathbb{K} }

\newcommand{\rt}{\rightarrow}

\newcommand{\I}{\mathbb{I} }

\newcommand{\image}{\operatorname{image}}

\newcommand{\Ass}{\operatorname{Ass}}

\newcommand{\depth}{\operatorname{depth}}
\newcommand{\Spec}{\operatorname{Spec}}

\newcommand{\height}{\operatorname{height}}

\newcommand{\Mod}{\operatorname{Mod}}

\newcommand{\Min}{\operatorname{Min}}
\newcommand{\Hom}{\operatorname{Hom}}

\newcommand{\Supp}{\operatorname{Supp}}

\theoremstyle{plain}

\newtheorem{theorem}{Theorem}[section]

\newtheorem{lemma}[theorem]{Lemma}
\newtheorem{proposition}[theorem]{Proposition}

\theoremstyle{definition}

\newtheorem{s}[theorem]{}

\theoremstyle{remark}

\begin{document}

\title{On functors that detect  $S_n$}
 \author{Tony J. Puthenpurakal}
\date{\today}
\address{Department of Mathematics, Indian Institute of Technology Bombay, Powai, Mumbai 400 076, India}
\email{tputhen@math.iitb.ac.in}
\begin{abstract}
Let $A$ be a Noetherian ring. For each $k$ where $0 \leq k \leq \dim A$ we construct left exact functors $D_k$ on $\Mod(A)$. Let $D^i_k$ be the $i^{th}$-right derived functor of $D_k$.  Let $M$ be a finitely generated $A$-module. Under mild conditions on $A$ and $M$ we prove that vanishing of some finitely many $D^i_k(M)$ is equivalent to $M$ satisfying  $S_n$.
\end{abstract}

\maketitle

\section{introduction}
Let $A$ be a Noetherian ring and let $M$ be a finitely generated $A$-module. Let $n \geq 0$ be a non-negative integer. Recall that $M$ satisfies $S_n$ if
\[
\depth M_\p \geq \min\{n , \dim M_\p \} \quad \text{for all primes} \ \p \ \text{in} \ A.
\]
Note that by convention the zero module has depth $+ \infty$ and dimension $-1$.
In this paper we construct functors which  (under mild conditions) detect whether $M$ satisfies $S_n$.

Let $E$ be a not-necessarily finitely generated $A$-module. By $\dim E$ we mean dimension of the support of $E$ considered as a subspace of 
$\Spec A$.  Let $k \geq 0$ be an integer.
Set 
\[
D_k(E)  = \sum_{\substack{N \ \text{submodule of } \ E \ \\ \   \dim N \leq k}} N
\] 
Clearly $D_k(E)$ is a submodule of $E$. Also if $\phi \colon E \rt F$ is $A$-linear then it is easy to verify that $\phi(D_k(E)) \subseteq D_k(F)$.
Set $D_k(\phi) \colon D_k(E) \rt D_k(F)$ to be the restriction of $\phi$ on $D_k(E)$.
Clearly we have an additive  functor $D_k$ on $\Mod(A)$. It can be shown that $D_k$ is left exact; see section 2. Let $D^i_k$ be the $i^{th}$-right derived functor of $D_k$.
 
 To prove our results we need to assume that the ring $A$ satisfies certain conditions.
 
 \s \label{assump}  We assume that $A$  satisfies the following properties:
 
 \begin{enumerate}
 \item
 $ \dim A$ is finite.
 \item
  $A$ is catenary.
 \item
 $A$ is equi-dimensional, i.e.,  $\dim A/\p = \dim A$ for all minimal primes $\p$ of $A$.
 \item
 If $\m$ is a maximal ideal in $A$ and $\p$ is a minimal prime of $A$ then 
 \\ $\height(\m/\p) = \dim A$. 
 \end{enumerate}
 
We now give examples of rings which satisfy the hypotheses in \ref{assump}:
\begin{enumerate}[ \rm(i)]

\item

$A =  R/I$ where $R = K[X_1,\cdots, X_n]$  and  $I$ is an equi-dimensional ideal in  $R$, i.e., $\height \p  = \height I$ for all minimal primes $\p$ of $I$.

\item
$A = R/I$ where $R =  \mathcal{O}[X_1,\ldots,X_n]$; $\mathcal{O}$ is the ring of integers in a number field (i.e., a finite extension of $\mathbb{Q}$) and $I$ is an unmixed ideal of $R$.

\item

$A = R/I$  where $R$ is a \CM \  local ring and $I$ is an equi-dimensional ideal. 

\item

$A$ is a catenary local domain.

\end{enumerate}

Recall a finitely generated   $A$-module $M$ is said to be equi-dimensional if $\dim M$ is finite and $\dim A/\p = \dim M$ for all minimal primes of $M$.
Our main result is

\begin{theorem} \label{main}
Let $A$ be a Noetherian ring satisfying the hypotheses in \ref{assump} and let $M$ be a finitely generated equi-dimensional $A$-module of dimension $\geq 1$. Let $n$ be an integer between $1$ and $\dim M$. Then the following conditions are equivalent:
\begin{enumerate}[\rm (i)]
\item
$M$ satisfies $S_n$.
\item
 $D^i_k(M) = 0$ for  $i = 0,1,\ldots, n-1$ and $0 \leq k < \dim M - i$.
  \end{enumerate}
  \end{theorem}
  
Here is an overview of the contents of the paper. In section two we define our functors $D_k$ and prove a few basic properties. In section three we prove a crucial result regarding localization of our functors $D_k$. Finally in section four we prove Theorem \ref{main}.

\section{The functors $D_k$}
In this section we define the functors $D_k$ and prove some of its basic properties.
Throughout $A$ is a Noetherian ring. The $A$-modules considered in this section need not be finitely generated .
\s Let $E$ be a $A$-module. Let $\Supp E$ denote the support of $E$. Set $\dim E = \dim \Supp E$. The following result is well-known
\begin{proposition}\label{basic-prop}
Let $0 \rt E_1 \rt E_2 \rt E_3 \rt 0$ be an exact sequence of $A$-modules. Then
\begin{enumerate}[\rm (a)]
\item
$\Supp E_2 = \Supp E_1 \cup \Supp E_3$.
\item
$\dim E_2 = \max \{ \dim E_1, \dim E_3 \}$.
\end{enumerate}
\end{proposition}

\s We now define our functors $D_k$. Let $k \geq 0$ be an integer. Let $E$ be an $A$-module. Set 
\[
D_{k,A}(E)  = \sum_{\substack{N \ \text{submodule of } \ E \ \\ \   \dim N \leq k}} N.
\] 
We suppress $A$ in $D_{k,A}(E)$ if it is clear from the context. 
Clearly $D_k(E)$ is a submodule of $E$. The following Lemma is useful.
\begin{lemma}\label{bl}
Let $\xi \in D_k(E)$. Then there exists a finitely generated $A$-submodule $M$  of $E$
with $\xi \in M$ and $\dim M \leq k$.
\end{lemma}
\begin{proof}
There exists $A$-submodules $N_1, \cdots, N_s$ of $E$ with 
$\dim N_i \leq k$ and 
$\xi = n_1 + n_2 + \cdots + n_s$ where $n_i \in N_i$.

Set $N = N_1 + N_2 + \cdots + N_s$. There is a natural surjective map 
$\bigoplus_{i=1}^{s} N_i \rt N$. By \ref{basic-prop}   it follows that $\dim N \leq k$. Also $\xi \in N$.

Set $M = A\xi \subseteq N$. By \ref{basic-prop}   it follows that $\dim M \leq k$. Also $\xi \in M$.
\end{proof}
\begin{proposition}\label{map}
Let $\phi \colon E \rt F$ be $A$-linear. Then $\phi(D_k(E)) \subseteq D_k(F)$.
\end{proposition}
\begin{proof}
Let $\xi \in D_k(E)$. Then  by Lemma \ref{bl}   there exists  a finitely generated $A$-submodule $N$ of $E$ with 
$\dim N\leq k$ and 
$\xi \in N$.
Then $\phi(\xi) \in \phi(N)$. Clearly $\phi(N)$ is an $A$-submodule of $F$.  Furthermore $\phi$ induces  a surjective map $N \rt \phi(N)$. By \ref{basic-prop} we get that $\dim \phi(N) \leq k$. Thus $\phi(\xi) \in D_k(F)$.
\end{proof}

\s
Set $D_k(\phi) \colon D_k(E) \rt D_k(F)$ to be the restriction of $\phi$ on $D_k(E)$.
Clearly we have an additive  functor $D_k$ on $\Mod(A)$. We show
\begin{proposition}\label{left}
$D_k$ is left exact.
\end{proposition}
\begin{proof}
Let $0 \rt E \xrightarrow{\alpha} F \xrightarrow{\beta} G$ be an exact sequence. We want to prove that the sequence 
$$0 \rt D_k(E) \xrightarrow{D_k(\alpha)} D_k(F) \xrightarrow{D_k(\beta)} D_k(G),$$
is exact.

Clearly $D_k(\alpha)$ is injective. Also  
$$D_k(\beta) \circ D_k(\alpha) = D_k(\beta \circ \alpha) = D_k(0) = 0,$$
as $D_k$ is an additive functor. Therefore $\image D_k(\alpha) \subseteq \ker D_k(\beta)$.

Let $\xi \in \ker D_k(\beta)$. In particular $\xi \in \ker \beta$. So there exists $e \in E$ with $\alpha(e) = \xi$. 
As $\xi \in D_k(F)$, 
 by Lemma \ref{bl}   there exists  a finitely generated $A$-submodule $N$ of $F$ with 
$\dim N\leq k$ and 
$\xi \in N$.
Note that $\alpha$ induces an exact sequence
\[
0 \rt \alpha^{-1}(N) \rt N.
\]
 By \ref{basic-prop} we get that $\dim \alpha^{-1}(N) \leq k$. Also $e \in \alpha^{-1}(N)$. It follows that $e \in D_k(E)$. Thus $D_k$ is left exact.
\end{proof}

We need the following two properties of $D_k$.
\begin{proposition}\label{dir-sum}
\begin{enumerate}[\rm (a)]
\item
Let $E$ be an $A$-module and let $L$ be an $A$-submodule of $E$. Then
$D_k(L) = D_k(E)\cap L$.
\item
Let $E_\alpha$ be a family of $A$-modules with $\alpha \in \Gamma$. Then
\[
D_k\left( \bigoplus_{\alpha \in \Gamma} E_{\alpha} \right) =  \bigoplus_{\alpha \in \Gamma} D_k(E_{\alpha}).
\]
\end{enumerate}
\end{proposition}
\begin{proof}
(a) Clearly $D_k(L) \subseteq D_k(E)\cap L$. Let $\xi \in D_k(E)\cap L$. By \ref{bl}
there exists a finitely generated $A$-submodule $N$ of $E$ with $\dim N \leq k$ and $\xi \in N$. So $\xi \in N \cap L$. By \ref{basic-prop}, $\dim N \cap L \leq \dim N \leq k$. So $\xi \in D_k(L)$.

(b) As $D_k$ is an additive functor the result holds if $\Gamma$ is a finite set. 

It is clear that  
\[
\bigoplus_{\alpha \in \Gamma} D_k(E_{\alpha}) \subseteq D_k\left( \bigoplus_{\alpha \in \Gamma} E_{\alpha} \right).
\]
Let $\xi \in D_k\left( \bigoplus_{\alpha \in \Gamma} E_{\alpha} \right)$. By \ref{bl}
there exists a finitely generated $A$-submodule $N$ of $ \bigoplus_{\alpha \in \Gamma} E_{\alpha}$ with $\dim N \leq k$ and $\xi \in N$.
Say
\[
\xi = \sum_{i=1}^{s}\xi_{\alpha_i} \quad \text{with} \ \xi_{\alpha_i} \in M_{\alpha_i}.
\]
Then $\xi \in N^\prime$ where $N^\prime = N \cap( \bigoplus_{i =1}^{s}M_{\alpha_i})$.
By \ref{basic-prop} $\dim N^\prime \leq k$. So
\[
\xi \in D_k\left(\bigoplus_{i =1}^{s}M_{\alpha_i}\right) =  \bigoplus_{i =1}^{s}D_k(M_{\alpha_i}) \subseteq \bigoplus_{\alpha \in \Gamma} D_k(E_{\alpha}).
\]
\end{proof}

We will also need the following computation.   
\begin{lemma}\label{comput}
Assume $\dim A $ is finite. Let $\q$ be a prime ideal in $A$ and let $E(A/\q)$ is the injective hull of $A/\q$. Then
\[
D_k(E(A/\q)) = \begin{cases} E(A/\q), &\text{if $\dim A/\q \leq k$} \\
               0, &\text{otherwise.}  \end{cases} 
\]
\end{lemma}
\begin{proof}
Let $N$ be a non-zero finitely generated $A$-submodule of $E(A/\q)$. Let $\p$ be a minimal prime of $N$ with $\dim A/\p = \dim N$. Note $\p \in \Ass N \subseteq \Ass E(A/\q) = \{ \q \}$. So $\p = \q$. It follows that $\dim N = \dim A/\q$.
As a consequence we have that $D_k(E(A/\q)) = 0$ if $\dim A/\q > k$.

Now assume $\dim A/\q \leq k$. Let $\xi \in E(A/\q)$ be non-zero. Set $N = A\xi$. 
Then $\dim N = \dim A/\q \leq k$. So $\xi \in D_k(E(A/\q))$. It  follows that
$D_k(E(A/\q)) = E(A/\q)$ if $\dim A/\q \leq k$.
\end{proof}

\section{Localization}
In this section we assume that $A$ satisfies our assumptions \ref{assump}. The goal of this section is to prove the following:
\begin{theorem}\label{local}
Assume $A$ satisfies \ref{assump}. Let $M$ be an $A$-module and let $\p$ be a prime ideal in $A$. Set $r = \dim A/\p$. Then for all $k \geq 0$ we have
\[
D^i_{k+r, A}(M)_\p \cong D^i_{k,A_\p}(M_\p) \quad \text{for all }  \ i \geq 0.
\]
\end{theorem}
To prove Theorem \ref{local} we need several  preparatory results. We first prove:
\begin{lemma}\label{ht-prep}
Assume $A$ satisfies \ref{assump}. Let $\p, \q$ be prime ideals in $A$ with $\q \subseteq \p$. Then
\begin{equation*}
\dim A/\q = \height (\p/\q) + \dim A/\p =  \dim A_\p/\q A_\p + \dim A/\p. 
\end{equation*}
\end{lemma}
\begin{proof}
It is easy to see that if $\m$ is a maximal ideal of $A$ then $A_\m$ satisfies the conditions of \ref{assump}.
We also get
\begin{equation*}
\dim A/\p + \height \p = \dim A. \tag{$\dagger$}
\end{equation*}

 We first note the following:
if $ \p_0 \subseteq \p_1 \subseteq \cdots \p_r = \p$ is a saturated chain of prime ideals with $\p_0$ a minimal prime then $r = \height \p$. To see this extend it to a maximal chain
$ \p_0 \subseteq \p_1 \subseteq \cdots \subseteq \p_s = \m$ where $\m$ is a maximal ideal in $A$. Then by assumption on $A$ we get $s = \dim A$. Localize at $\m$. Then by \cite[Lemma 2, p.\ 250]{Mat} we get that $\height \m = \height \p_\m + \height( \m/\p)$.  Note $\height \p_\m = \height \p \geq r$ and $\height(\m/\p) \geq s-r$. As  $\height \m = \dim A = s$ we get that 
$r = \height \p$ and $s-r = \height(\m/\p)$.  It is now elementary to see that $\height(\p/\q) = \height \p - \height \q$.

Note that by $(\dagger)$ we get  $\dim A/\q - \dim A/\p = \height \p - \height \q$.
The result follows.

\end{proof}
\begin{lemma}\label{prep-lemma}
Assume $A$ satisfies \ref{assump}. Let $\p$ be a prime ideal in $A$. Set $r = \dim A/\p$. Let $\q$ be a prime ideal in $A$ with $\q \subseteq \p$. Let $k \geq 0$.
Then
\[
D_{k+r, A}(E_A(A/\q))  \cong  D_{k+r, A}(E(A/\q))_\p   \cong  D_{k,A_\p}(E_{A_\p}(A_\p/\q A_\p)).
\]
\end{lemma}
\begin{proof}
As $A$ satisfies \ref{assump}, by \ref{ht-prep} we get 
\begin{equation*}
\dim A/\q = \height (\p/\q) + \dim A/\p =  \dim A_\p/\q A_\p + r. \tag{*}
\end{equation*}

To prove our result we consider two cases.

Case 1: $\dim A/\q \leq k +r$. \\
By $(*)$ this holds if and only if $\dim A_\p/\q A_\p \leq k$.
By 
Lemma \ref{comput}  we have
\[
D_{k+r,A}(E_A(A/\q)) = E_A(A/\q) \quad \text{and} \quad 
D_{k,A_\p}(E_{A_\p}(A_\p/\q A_\p))  = E_{A_\p}(A_\p/\q A_\p).
\]
The result follows since $E_A(A/\q)  \cong E_A(A/\q)_\p  \cong E_{A_\p}(A_\p/\q A_\p)$.

Case 2. $\dim A/\q > k + r$. \\
By $(*)$ this holds if and only if $\dim A_\p/\q A_\p > k$.  By Lemma \ref{comput} we have
\[
D_{k+r, A}(E_A(A/\q)) = 0 \quad \text{and} \quad
D_{k, A_\p}(E_{A_\p}(A_\p/\q A_\p))  = 0.
\]
The result follows.
\end{proof}

We now show:
\begin{proposition}\label{loc-basic}
Assume $A$ satisfies \ref{assump}. Let $\p$ be a prime ideal in $A$. Set $r = \dim A/\p$. Let $M$ be an $A$-module. Let $k \geq 0$.  Then
\[
D_{k+r, A}(M)_\p \cong D_{k,A_\p}(M_\p).
\]
\end{proposition}
\begin{proof}
We consider two cases.

Case 1: $M$ is an injective $A$-module. By Matlis theory, cf. \cite[18.5]{Mat}
\[
M = \bigoplus_{\q \in \Spec A} E_A(A/\q)^{\mu_\q}.
\]
Notice $\mu_\q  =  \dim_{\kappa(\q)} \Hom_{A_\q}(\kappa(\q), M_\q)$ (here $\kappa(\q)$  is the residue field of $A_\q$).
By Proposition \ref{dir-sum} we have
\[
D_{k+r,A}(M) = \bigoplus_{\q \in \Spec A}D_{k+r,A}( E_A(A/\q))^{\mu_\q}.
\]
Now note that 
\[
M_\p =  \bigoplus_{\q \subseteq \p} E_{A_\p}(A_\p/\q A_\p)^{\mu_\q}.
\]
Therefore by Proposition \ref{dir-sum} we get that
\[
D_{k,A_\p}(M_\p) = \bigoplus_{\q \subseteq \p} D_{k,A_\p}(E_{A_\p}(A_\p/\q A_\p))^{\mu_\q}.
\]
The result now follows from
Proposition \ref{prep-lemma}.

Case 2:  $M$ is an arbitrary $A$-module. 

Embed $M$ into an injective $A$-module $I$. Then note that $M_\p$ is a submodule of $I_\p$.

By Proposition \ref{dir-sum} we get $D_{k+r,A}(M) = D_{k+r, A}(I) \cap M$. So  we get 
\begin{align*}
D_{k+r,A}(M)_\p &= (D_{k+r,A}(I) \cap M)_\p, \\
&\cong  D_{k+r,A}(I)_\p \cap M_\p; \  \  \  \text{by \cite[7.4(i)]{Mat}},\\
&\cong  D_{k,A_\p}(I_\p) \cap M_\p;  \  \  \ \text{by Case 1},\\
&=   D_{k,A_\p}(M_\p);  \  \  \ \text{by Proposition \ref{dir-sum}}.
\end{align*}

\end{proof}
We now give
\begin{proof}[Proof of Theorem \ref{local}]
Let $\I$ be a minimal injective resolution of $M$. Then note that $\I_\p$ is a minimal injective resolution of $M_\p$, \cite[Lemma 6, p.\ 149]{Mat}. Consider the complex $\D = D_{k+r,A}(\I)$. By \ref{loc-basic} we get that $\D_\p = D_{k,A_\p}(\I_\p)$. As $\D$ is a complex of injectives,  the map $\D \rt \D_\p$ is a surjective map of complexes. So we have an exact 
sequence of complexes $0 \rt \K \rt \D \rt \D_\p \rt 0$. Observe that by \ref{comput}, 
$$ \K^i = \bigoplus_{ \substack{\q \nsubseteq \p \\ \dim A/\q \leq r + k} }E_A(A/\q)^{\mu_i(M,\q)}. $$
It follows that $\K_\p = 0$.

The short exact sequence of complexes $ 0 \rt \K \rt \D \rt \D_\p \rt 0$ yields  a long exact sequence
\[
 \cdots \rt H^i(\K) \rt H^i(\D) \rt H^i(\D_\p) \rt H^{i+1}(\K) \cdots 
\]
As $K_\p = 0$ we get that $H^i(\K)_\p = 0$ for all $i$. Thus $H^i(\D)_\p \cong H^i(\D_\p)_\p = H^i(\D_\p)$.
The result follows.
\end{proof}
\section{Proof of Theorem \ref{main}}
In this section we prove Theorem \ref{main} by induction on $n$. We prove the base case $n = 1$ separately.
\begin{proposition}\label{base}
Assume $A$ satisfies \ref{assump}. Let $M$ be a finitely generated equidimensional $A$-module of dimension $\geq 1$. The following conditions are equivalent:
\begin{enumerate}[\rm (i)]
\item
$M$ satisfies $S_1$.
\item
$D_k(M) = 0 $ for all $k < \dim M$.
\end{enumerate}
\end{proposition}
\begin{proof}
We first assume $M$ satisfies $S_1$. Then $\dim M_\p \geq 1$ if and only if $\p \notin  \Ass M$. Suppose $\xi \in D_k(M)$ is non-zero. Then by \ref{bl}  there exists a finitely generated submodule $N$ of $M$ with $\dim N \leq k$ and $\xi \in N$. Let $\p \in \Ass N$ be such that $\dim A/\p = \dim N$. Note $\p \in \Ass M$. It follows that
$\p \in \Min M$. So $\dim N = \dim M$. It follows that $k \geq \dim M$. Thus $D_k(M) = 0$ for $k < \dim M$.

Conversely assume that $D_k(M) = 0 $ for all $k < \dim M$. Suppose if possible  $M$ does not satisfy $S_1$. Then there exists $\p$ with $\dim M_\p \geq 1$ and $\depth M_\p = 0$. Thus $\p \in \Ass M$. So we have an injection $A/\p \rt M$. Notice $c = \dim A/\p < \dim M$. Thus $D_c(M) \neq 0$, a contradiction.
\end{proof}

We now give
\begin{proof}[Proof of Theorem \ref{main}]
We prove the result by induction on $n$. We have proved the result for $n = 1$, see \ref{base}. We assume the result for $n-1 \geq 1$ and prove it for $n$.

We first assume that $M$ satisfies $S_n$-property. As $M$ also satisfies $S_{n-1}$ we get by induction hypothesis that $D^j_k(M) = 0$ for $k < \dim M - j$ and $j = 0,1,\cdots, n-2$
Let $\I$ be a minimal injective resolution for $M$. As $M$ satisfies $S_n$ we get that  for $i \leq n-1$,
\[
\I^i = \bigoplus_{\dim M_\p \leq i} E(A/\p)^{\mu(\p, M)}.
\]
Suppose $\xi \in D_k(\I^{n-1})$ is non-zero. Then by \ref{bl} there exists a finitely generated $A$-submodule $N$ of $\I^{n-1}$ with $\dim N \leq k$ and $\xi \in N$. Let $\p$ be a minimal prime of $N$ with $\dim A/\p = \dim N \leq k$. Then $\p \in \Ass \I^{n-1}$.
So $\dim M_\p \leq n -1$. Let $\q$ be a minimal prime of $M$ contained in $\p$. Then
by \ref{ht-prep} we get
$$\dim M = \dim A/\q = \dim A_\p /\q A_\p + \dim A/\p \leq \dim M_\p  + \dim N \leq n -1 + \dim N. $$
So $\dim N \geq \dim M - n +1$. It follows that $D_k(\I^{n-1}) = 0$ for $k < \dim M - n+1$. Thus $D^{n-1}_k(M) = 0$ for $k < \dim M - n + 1$. 

We now assume that $D^i_k(M) = 0$ for  $i = 0,1,\ldots, n-1$ and $0 \leq k < \dim M - i$. By induction hypotheses it follows that $M$ satisfies $S_{n-1}$. Suppose if possible $M$ does not satisfy $S_{n}$. Then there exists a prime ideal $\p$ with $\dim M_\p \geq n$ and $\depth M_\p = n-1$. We localize at $\p$. We get that $D_{0,A_\p}^{n-1}(M_\p) \neq 0$. By Theorem \ref{local} it follows that $D_r^{n-1}(M) \neq 0$ where $r = \dim A/\p$. 

Claim:  $\dim M = \dim M_\p + r$.\\
Assume the claim for the moment. Then $r = \dim M - \dim M_\p \leq \dim M - n < \dim M - n + 1$. Also $D_r^{n-1}(M) \neq 0$. This contradicts our assumption.

\textit{Proof of claim}. Let $\q$ be a minimal prime of $M$ contained in $\p$ and let $\m$ be an arbitrary maximal ideal of $A$ containing $\p$. By \ref{ht-prep} we get that $\dim M = \dim A/\q = \height(\m/\q)$. As $A$ is catenary we get that $\height(\m/\q) = \height(\m/\p) + \height(\p/\q)$. We take $\q$ with $\height(\p/\q) = \dim M_\p$. Also note that again by \ref{ht-prep}, $\height(\m/\p) = \dim A/\p = r$.
\end{proof}


\begin{thebibliography}{10}
\bibitem{Mat}
H.~Matsumura, 
\emph{Commutative ring theory},
Translated from the Japanese by M. Reid. Second edition. Cambridge Studies in Advanced Mathematics, 8. Cambridge University Press, Cambridge, 1989. 

\end{thebibliography}
\end{document}